\documentclass{amsart}
\usepackage{amssymb,color,latexsym,amsmath,amsthm}
\usepackage{fullpage}
\usepackage{hyperref}
\usepackage{titletoc}

\usepackage[numbers,sort&compress]{natbib}

\numberwithin{equation}{section}
\newtheorem{theorem}{Theorem}[section]
\newtheorem{proposition}[theorem]{Proposition}
\newtheorem{corollary}{Corollary}[section]

\newtheorem*{thm*}{Theorem A}
\newtheorem{thm}{Theorem}[section]
\newtheorem{dfn}{Definition}[section]

\newtheorem{lemma}{Lemma}[section]

\theoremstyle{definition}

\def\XXint#1#2#3{{\setbox0=\hbox{$#1{#2#3}{\int}$}
     \vcenter{\hbox{$#2#3$}}\kern-.5\wd0}}

\def\ep{\varepsilon}

\def\l{\lambda}
\def\R{{\mathbb R} }

\def\r{\R^{n+1}_{+}}
\def\rn{\R^{n}}
\def\B{\mathbb{R}^n}
\def\br{\partial\r}
\def\super{\overline}
\def\s{(-\Delta)^s}
\def\ss{(-\Delta)^{\frac{s}{2}}}
\def\b{\mathbb{S}^{n-1}}
\def\ou{\overline{u}}
\def\D{\Delta}
\def\R{{\mathbb R} }

\def\bg{\begin{equation*}}
\def\ng{\end{equation*}}
\def\bge{\begin{eqnarray*}}
\def\nge{\end{eqnarray*}}
\def\r{\R^{n+1}_{+}}
\def\rn{\R^{n}}
\def\br{\partial\r}
\def\super{\overline}
\def\of{\overline{f}_\alpha}

\def\ofa{\overline{f}_{\alpha-1}}

\def\ofb{\overline{f}_{\alpha+1}}

\def\f{f_\alpha}

\def\fa{f_{\alpha-1}}

\def\fb{f_{\alpha+1}}

\def\ov{\overline{V}_\alpha}

\def\bff{{\bf f}}

\begin{document}

\title[fractional Toda system]{On stable and finite Morse index solutions of the fractional Toda system}

\author[M. Fazly]{Mostafa Fazly}
\address{\noindent Mostafa Fazly, Department of Mathematics, University of Texas at San Antonio, San Antonio TX 78249.}
\email{mostafa.fazly@utsa.edu}

\author[W. Yang]{ Wen Yang}
\address{\noindent Wen ~Yang,~Wuhan Institute of Physics and Mathematics, Chinese Academy of Sciences, P.O. Box 71010, Wuhan 430071, P. R. China}
\email{wyang@wipm.ac.cn}

\begin{abstract}
We develop a monotonicity formula for solutions of the
fractional Toda system
$$   (-\Delta)^s f_\alpha    =   e^{-(f_{\alpha+1}-f_\alpha)}  -  e^{-(f_\alpha-f_{\alpha-1})}  \quad  \text{in} \ \  \mathbb R^n,$$
when $0<s<1$,  $\alpha=1,\cdots,Q$, $f_0=-\infty$, $f_{Q+1}=\infty$ and $Q \ge2$ is the number of equations in this system.  We then apply this formula, technical integral estimates, classification of stable homogeneous solutions, and blow-down analysis arguments to establish  Liouville type theorems for finite Morse index (and stable) solutions of the above system
when  $n > 2s$ and
$$
\dfrac{\Gamma(\frac{n}{2})\Gamma(1+s)}{\Gamma(\frac{n-2s}{2})} \frac{Q(Q-1)}{2}  > \frac{ \Gamma(\frac{n+2s}{4})^2  }{ \Gamma(\frac{n-2s}{4})^2} .
$$ Here, $\Gamma$ is the Gamma function.  When $Q=2$, the above equation is the classical (fractional) Gelfand-Liouville equation.

\end{abstract}

\maketitle

\noindent
{\it \footnotesize 2010 Mathematics Subject Classification  35B45, 35B08, 45K05, 45G15, 35B53}. {\scriptsize }

\section{Introduction}
In this article, we study stable and finite Morse index solutions of the fractional Toda system
 \begin{equation} \label{main}
   (-\Delta)^s f_\alpha    =   e^{-(f_{\alpha+1}-f_\alpha)}  -  e^{-(f_\alpha-f_{\alpha-1})}  \quad  \text{in} \ \  \mathbb R^n, \quad \alpha=1,\cdots,Q,
  \end{equation}
where $0<s<1$ and $Q \ge2$ is the number of equations in this system. Let $f_0=-\infty$ and $f_{Q+1}=\infty$.  By taking $u_\alpha:=f_\alpha-f_{\alpha+1}$, we can write \eqref{main} into the following classical form
\begin{equation}
\label{1.otoda}
(-\Delta)^su_\alpha=2e^{u_\alpha}-e^{u_{\alpha+1}}-e^{u_{\alpha-1}}\quad  \text{in} \ \  \mathbb R^n,\quad \alpha=1,\cdots,Q-1,
\end{equation}
where $u_0=u_Q=-\infty.$ Equation \eqref{1.otoda} is called the $SU(Q)$ Toda system. When $Q=2$, the above equation \eqref{1.otoda} is the classical (fractional) Gelfand-Liouville equation,
 \begin{equation} \label{maingl}
   (-\Delta)^s u    =   e^{u} \quad  \text{in} \ \  \mathbb R^n.
  \end{equation}
The Toda system (in classical setting) and the Gelfand-Liouville equation have been studied in the research areas rooted in mathematical analysis and geometry.
In physics, the Toda system is connected with the Chern-Simons-Higgs system, while in complex geometry,  solutions of the Toda system (in classical setting) are closely related to the holomorphic curves in projective spaces. Particularly, the classical Pl\"ucker formula can be written as a local version of the $SU(Q)$ Toda system. The literature in this
context is too vast to give more than a few representative references,  see \cite{m5, bennet,doliwa,dunne1,mal,lwy,ko,jw,dkw,lwyz,lwz0,lwz,lyz} for the background and recent developments. In this paper, we shall study \eqref{main} from another point of view, i.e., focusing on classifying the stable and finite Morse index solutions of \eqref{main}. Stable and  finite Morse index solutions of supercritical fractional elliptic equations
\begin{equation}
(-\Delta)^s u=W(u) \ \ \mathbb R^n,
\end{equation}
including Gelfand-Liouville equation and Lane-Emden equation  when $W(u)=e^u$ and $W(u)=u^p$ for $p>1$, respectively, are studied in the literature as well, see \cite{d,df,ddw,dggw,fw,fwy,hy,ww}. In the absence of stability, the classification of solutions of  Gelfand-Liouville equation in lower-dimensions is studied in \cite{clo,dmr} and references therein. Notice that the classification of stable solutions of elliptic multi-component systems is also studied in the context of De Giorgi's conjecture, see \cite{fs,kwang2,st} and references therein.

We assume that $f_\alpha \in C^{2\gamma}(\mathbb R^n)$, $\gamma >s>0$ and
\begin{equation}
\int_{\mathbb R^n} \frac{|f_\alpha(x)|}{(1+|x|)^{n+2s}} dx<\infty.
\end{equation}
 The fractional Laplacian of $u$ when $0<s<1$  denoted by
\begin{equation}
(-\Delta)^s f_\alpha(x):= p.v. \int_{\mathbb R^n} \frac{f_\alpha(x)-f_\alpha(z)}{|x-z|^{n+2s}} dz,
\end{equation}
is well-defined for every $x\in\mathbb R^n$. Here $p.v.$ stands for the principle value. It is by now standard that the fractional Laplacian can be seen as a Dirichlet-to-Neumann operator for a
degenerate but local diffusion operator in the higher-dimensional half-space $\mathbb R^{
n+1}$, see Caffarelli and Silvestre in \cite{cs}.  In other words, for $f_\alpha\in C^{2\gamma} \cap L^1(\mathbb R^n, (1+|x|^{n+2s})dx)$ when $\gamma>s$ and $0<s<1$,  there exists  $ \of\in C^2(\mathbb R^{n+1}_+)\cap C(\super{\mathbb R^{n+1}_+})$ such that $y^{1-2s}\partial_{y} \of \in C(\super{\mathbb R^{n+1}_+})$ and

\begin{equation}\label{maine}
\left\{
\begin{aligned}
\nabla\cdot(y^{1-2s}\nabla  \of)&=0&\quad&\text{in $\R^{n+1}_{+}$,}\\
  \of&=  f_\alpha&\quad&\text{on $\br$,}\\
-\lim_{y\to0} y^{1-2s}\partial_{y} \of&= \kappa_s   (-\Delta)^s f_\alpha  &\quad&\text{on $\br$,}
\end{aligned}
\right.
\end{equation}
for the following constant $\kappa_s$,
\begin{equation} \label{kappas}
\kappa_s := \frac{\Gamma(1-s)}{2^{2s-1} \Gamma(s)}.
\end{equation}
Here is the notion of stability.
 \begin{dfn}
We say that a solution $\bff=(f_1,\cdots,f_Q)$ of (\ref{main}) is stable outside a compact set  if there exists $R_0>0$ such that
\begin{equation}\label{stability}
\frac{C_{n,s}}{2} \sum_{\alpha} \int_{\mathbb R^n} \int_{\mathbb R^n} \frac{ ( \phi_\alpha(x)-\phi_\alpha(y) )^2 }{|x-y|^{n+2s}} dx dy= \sum_{\alpha}\int_{\mathbb R^n}| (-\Delta )^{\frac{s}{2}}\phi_\alpha|^2dx
\ge \sum_{\alpha}\int_{\mathbb R^n}  e^{-(f_\alpha-f_{\alpha-1})} (\phi_\alpha - \phi_{\alpha-1})^2 ,
\end{equation}
 for any $\phi\in C_c^\infty(\mathbb R^n\setminus \overline {B_{R_0}})$. We say a solution is stable if $\phi\in C_c^\infty(\mathbb R^n)$.
 \end{dfn}
From any solution to equation \eqref{main}, it is straightforward to see that the summation of $f_\alpha$ is a $s$-harmonic function, we shall assume that any solution to \eqref{main} satisfy
\begin{equation}
\label{1.summation}
\sum_{\alpha=1}^Qf_\alpha=0.
\end{equation}
Therefore, it is natural to assume that test functions $\{\phi_\alpha\}_\alpha$ satisfy an additional constraint
\begin{equation}
\sum_{\alpha=1}^Q \phi_\alpha = 0.
\end{equation}
Throughout this paper, we use the convention that $\phi_0=0$. Here, we list  our main results.

\begin{thm}\label{thmtodam}
There is no entire  finite Morse index solution of fractional Toda system (\ref{main}) for $1\le \alpha \le Q$ for $Q\ge 2$ when $n > 2s$ and
\begin{equation}\label{assum2}
\dfrac{\Gamma(\frac{n}{2})\Gamma(1+s)}{\Gamma(\frac{n-2s}{2})} \frac{Q(Q-1)}{2}  > \frac{ \Gamma(\frac{n+2s}{4})^2  }{ \Gamma(\frac{n-2s}{4})^2} .
\end{equation}
Here,  $\Gamma$ is the Gamma function.
\end{thm}

\begin{thm}\label{thmtodaq}
There is no entire stable solution of fractional Toda system (\ref{main}) for $1\le \alpha \le Q$ for $Q\ge 2$ when $n\le 2s$ or $n > 2s$ and
(\ref{assum2}) holds.
\end{thm}

When $Q=2$, that (\ref{main}) turns into the Gelfand-Liouville equation,  and $s=1$, (\ref{assum2}) reads $2<N<10$, and  $N=10$ is known as the critical Joseph-Lundgren optimal dimension. In addition, the above results recover the Dancer-Farina's classification results in \cite{df,f}. For the cases of $0<s<1$ and $1<s \le 2$,  (\ref{assum2}) concurs with the ones given in \cite{hy,fwy,dggw}.

Now define the energy functional  for any  $\lambda>0$ and $x_{0}\in\br$  as
\begin{equation}
\label{energyEs1}
\begin{aligned}
E(\overline{ \bff },\lambda,x_0) :=~&   \lambda^{2s-n} \sum_{\alpha} \left(\frac12\int_{\r\cap B^{n+1}(x_0,\lambda)} y^{1-2s} \vert\nabla\of\vert^2\;dx\,dy - \kappa_{s} \int_{\br\cap B^{n+1}(x_0,\lambda)}  e^{-(\of - \ofa)}\;dx\right)\\
&- s \lambda^{2s-1-n}  \sum_{\alpha} (2\alpha-Q-1)\int_{\partial B^{n+1}(x_0,\lambda) \cap\r}y^{1-2s} [\of - (2\alpha-Q-1)s \log r ]\;d\sigma.
\end{aligned}
\end{equation}
Here is a monotonicity formula for solutions of \eqref{maine}  when $0<s<1$ that is our main tool to establish the above results.

 \begin{thm}\label{thmono1s}
Suppose that $0<s<1$.  Let $\of\in C^2(\r)\cap C(\super\r)$ be  a solution of \eqref{maine}  such that $y^{1-2s}\partial_{y} \of\in C(\super\r)$. Then, $E$ is a nondecreasing function of $\lambda$. Furthermore,
\begin{equation}
\frac{dE}{d\lambda} = \lambda^{2s-n+1}\int_{\partial B^{n+1}(x_0,\l)\cap\r}y^{1-2s}\sum_\alpha\left(\frac{\partial  \of}{\partial r}-\frac{(2\alpha-Q-1)s}{r}\right)^2\;d\sigma ,
\end{equation}
where $E $ provided in (\ref{energyEs1}).
\end{thm}
Before ending the introduction, we would like to mention some technical difficulties, among other things, arising from the Toda system as opposed to the scalar  Gelfand-Liouville equation. The right-hand side of (\ref{main}) changes sign and this  makes it more challenging to derive a priori estimates for $u_\alpha$, like the ones in \cite[Lemma 3.2]{hy} and \cite[Lemma 4.12]{fwy}.  Due to the same reason, the study of  the representation formula of $u_\alpha^\l$, see \eqref{4.7}, is more challenging. We overcome  this difficulty by giving a more refined estimation in the proofs. To be more precise, instead of providing a point-wise bound to the integration (without the constant part), we prove that it is bounded in an area with a positive measure, and this is sufficient for our proofs.

Here is how this article is structured. In Section \ref{sechs}, we classify homogeneous solutions of the form of $f_\alpha(r,\theta)=\psi_\alpha(\theta) + (2\alpha-Q-1)s \log r$. In Section \ref{secen}, we provide technical integral estimates for solutions of for solutions of (\ref{main}), including an integral representation formula for each $f_\alpha$ and also Moser iteration type arguments. In Section \ref{secbd},  we prove the above monotonicity formula for solutions of (\ref{main}) via applying rescaling arguments. Then,  we perform blow-down analysis arguments and prove the rest of main results.

\medskip
\begin{center}
List of Notations:
\end{center}

\begin{itemize}
\item $B_R^{n+1}$ is the ball centered at $0$ with radius $R$ in dimension $(n+1)$.
\smallskip
\item $B_R$ is the ball centered at $0$ with radius $R$ in dimension $n$.
\smallskip
\item $B^{n+1}(x_0,R)$ is  the ball centered at $x_0$ with radius $R$ in dimension $(n+1)$.
\smallskip
\item $B(x_0,R)$ is the ball centered at $x_0$ with radius $R$ in dimension $n$.
\smallskip
\item $X=(x,y)$  represents a point in $\mathbb{R}_+^{n+1}.$
\smallskip
\item $\overline{f}$ is $s$-harmonic extension of function $f$ on $\mathbb{R}_+^{n+1}$.
\item $C$ is  a generic positive constant which may change from line to line.
\item $C(r)$ is a positive constant depending on $r$ and may change from line to line.
\item $\sigma$ is the $n$-dimensional Hausdorff measure restricted to $\partial B^{n+1}(x_0,r)$.
\end{itemize}

\section{Homogeneous Solutions}\label{sechs}
 In this section, we examine homogenous solutions of the form $f_\alpha(r,\theta)=\psi_\alpha(\theta) + (2\alpha-Q-1)s \log r$. Note that the methods and ideas are inspired by the ones used in \cite{ddw,fw,ddww,kwang}, but  considerably different arguments are applied in the proofs.
 \begin{thm}\label{homog}
Let $n>2s$ and
\begin{equation}\label{assum3}
\dfrac{\Gamma(\frac{n}{2})\Gamma(1+s)}{\Gamma(\frac{n-2s}{2})} \frac{Q(Q-1)}{2} > \frac{ \Gamma(\frac{n+2s}{4})^2  }{ \Gamma(\frac{n-2s}{4})^2}.
\end{equation}
 Then,  there is no homogeneous stable solution of the form of $f_\alpha(r,\theta)=\psi_\alpha(\theta) + (2\alpha-Q-1)s \log r$ for (\ref{main}).
 \end{thm}
 \begin{proof}
 Let $(f_1,\cdots, f_Q)$ be a solution of \eqref{main}. For every $\alpha$ and for any radially symmetric function $\varphi \in C_c^\infty(\rn)$ we have
\begin{align*}
&\int_{\rn}\left[  e^{-\left(\psi_{\alpha}(\theta) -\psi_{\alpha-1}(\theta) \right) } -e^{-\left(\psi_{\alpha+1}(\theta) -\psi_{\alpha}(\theta) \right) }\right] e^{-2s\log|x|}\varphi dx
\\&=-\int_{\rn}(\psi_\alpha(\theta)+(2\alpha-Q-1)s\log|x|)(-\Delta)^s\varphi(x)dx\\
&=(Q+1-2\alpha)s\int_{\rn}\log|x| (-\Delta)^s\varphi(x)dx=\mathcal A_{n,s,\alpha,Q}\int_{\rn}\frac{\varphi}{|x|^{2s}}dx,
\end{align*}
where $$\mathcal A_{n,s,\alpha,Q} =  2^{2s-1}\dfrac{\Gamma(\frac{n}{2})\Gamma(1+s)}{\Gamma(\frac{n-2s}{2})}(2\alpha-1-Q).$$
Here, we used
\begin{equation*}
\int_{\rn}\psi_\alpha(\theta) \s\varphi dx=0\quad \mbox{for any radially symmetric function}~\varphi\in C_c^\infty(\rn) ,
\end{equation*}
and
\begin{equation*}
\s\log\frac{1}{|x|^{2s}}=A_{n,s}\frac{1}{|x|^{2s}}
=2^{2s}\dfrac{\Gamma(\frac{n}{2})\Gamma(1+s)}{\Gamma(\frac{n-2s}{2})}
\frac{1}{|x|^{2s}}.
\end{equation*}
Then,  we derive that
\begin{equation*}
0=\int_{\rn}(  e^{-\left(\psi_{\alpha}(\theta) -\psi_{\alpha-1}(\theta) \right) } -e^{-\left(\psi_{\alpha+1}(\theta) -\psi_{\alpha}(\theta) \right) }-\mathcal A_{n,s,\alpha,Q})\frac{\varphi}{|x|^{2s}}dx
\end{equation*}
Therefore,
\begin{equation*}
0=\int_0^\infty r^{n-1-2s}\varphi (r)\int_{\mathrm{S}^{n-1}} (  e^{-\left(\psi_{\alpha}(\theta) -\psi_{\alpha-1}(\theta) \right) } -e^{-\left(\psi_{\alpha+1}(\theta) -\psi_{\alpha}(\theta) \right) }-\mathcal A_{n,s,\alpha,Q})d\theta dr,
\end{equation*}
which implies
\begin{equation}\label{integ}
\int_{\mathbb{S}^{n-1}}  e^{-\left(\psi_{\alpha}(\theta) -\psi_{\alpha-1}(\theta) \right) } -e^{-\left(\psi_{\alpha+1}(\theta) -\psi_{\alpha}(\theta) \right) } d\theta=\mathcal  A_{n,s,\alpha,Q} |\b|.
\end{equation}
 Applying inductions arguments starting with $\alpha=1$, we get
\begin{equation}\label{integ}
\int_{\mathbb{S}^{n-1}}  e^{-\left(\psi_{\alpha+1}(\theta) -\psi_{\alpha}(\theta) \right) } d\theta= \lambda_{n,s,\alpha,Q}|\b|,
\end{equation}
 for all $\alpha$ when
 \begin{equation}
 \lambda_{n,s,\alpha,Q} =  2^{2s-1}\dfrac{\Gamma(\frac{n}{2})\Gamma(1+s)}{\Gamma(\frac{n-2s}{2})}\alpha(Q-\alpha).
 \end{equation}
 We now set a radially symmetric smooth cut-off function
\begin{equation*}
\eta(x)=\begin{cases}
1, \qquad &\mathrm{for}~|x|\leq 1,\\
\\
0, &\mathrm{for}~|x|\geq2,
\end{cases}
\end{equation*}
and set
$$\eta_\ep(x)=\left(1-\eta\left(\frac{2x}{\ep}\right)\right)\eta(\ep x).$$
It is not difficult to see that $\eta_\ep=1$ for $\ep<r<\ep^{-1}$ and $\eta_\ep=0$ for either $r<\frac{\ep}{2}$ or $r>\frac{2}{\ep}$. We test the stability condition \eqref{stability} on the function $\phi_\alpha(x)=c_\alpha r^{-\frac{n-2s}{2}}\eta_\ep(r)$ where $c_\alpha$ is a constant depending on $\alpha$ satisfying $\sum_{\alpha} c_\alpha=0$. Let $z=rt$ and note that
\begin{equation*}
\begin{aligned}
\int_{\rn}\frac{\phi_\alpha(x)-\phi_\alpha(z)}{|x-z|^{n+2s}}dz
&= c_\alpha r^{-\frac n2-s}\int_0^\infty\int_{\b}\frac{\eta_{\ep}(r)-t^{-\frac{n-2s}{2}}\eta_{\ep}(rt)}{(t^2+1-2t\langle \theta,\omega\rangle)^{\frac{n+2s}{2}}}t^{n-1}dtd\omega\\
&= c_\alpha r^{-\frac n2-s}\eta_{\ep}(r)\int_0^{\infty}\int_{\b}\frac{1-t^{-\frac{n-2s}{2}}}
{(t^2+1-2t\langle\theta,\omega\rangle)^{\frac{n+2s}{2}}}t^{n-1}dtd\omega\\
&\quad +c_\alpha r^{-\frac n2-s}\int_0^{\infty}\int_{\b}\frac{t^{n-1-\frac{n-2s}{2}}(\eta_{\ep}(r)-\eta_{\ep}(rt))}
{(t^2+1-2t\langle\theta,\omega\rangle)^{\frac{n+2s}{2}}}dtd\omega.
\end{aligned}
\end{equation*}
It is known that
\begin{equation*}
\Lambda_{n,s}=C_{n,s}\int_0^{\infty}\int_{\b}\frac{1-t^{-\frac{n-2s}{2}}}
{(t^2+1-2t\langle\theta,\omega\rangle)^{\frac{n+2s}{2}}}t^{n-1}dtd\omega,
\end{equation*}
where
\begin{equation}
 \Lambda_{n,s}:=2^{2s}\frac{ \Gamma(\frac{n+2s}{4})^2  }{ \Gamma(\frac{n-2s}{4})^2}.
 \end{equation}
Therefore,
\begin{equation*}
\begin{aligned}
C_{n,s}\int_{\rn}\frac{\phi_\alpha(x)-\phi_\alpha(z)}{|x-z|^{n+2s}}dz
=~&c_\alpha C_{n,s} r^{-\frac n2-s}\int_0^{\infty}\int_{\b}\frac{t^{n-1-\frac{n-2s}{2}}(\eta_{\ep}(r)-\eta_{\ep}(rt))}
{(t^2+1-2t\langle\theta,\omega\rangle)^{\frac{n+2s}{2}}}dtd\omega\\
&+c_\alpha \Lambda_{n,s}r^{-\frac n2-s}\eta_{\ep}(r).
\end{aligned}
\end{equation*}
Based on the above computations, we compute the left-hand side of the stability inequality \eqref{stability},
\begin{equation}
\label{3.exp-1}
\begin{aligned}
&\frac{C_{n,s}}{2}\int_{\rn}\int_{\rn}\frac{(\phi_\alpha(x)-\phi_\alpha(z))^2}{|x-z|^{n+2s}}dxdz\\
&=C_{n,s}\int_{\rn}\int_{\rn}\frac{(\phi_\alpha(x)-\phi_\alpha(z))\phi_\alpha(x)}{|x-z|^{n+2s}}dxdz\\
&=c_\alpha ^2 C_{n,s}\int_0^\infty\left[\int_0^\infty r^{-1}\eta_{\ep}(r)(\eta_{\ep}(r)-\eta_{\ep}(rt))dr\right]
\int_{\b}\int_{\b}\frac{t^{n-1-\frac{n-2s}{2}}}
{(t^2+1-2t\langle\theta,\omega\rangle)^{\frac{n+2s}{2}}}d\omega d\theta dt\\
&\quad+c_\alpha ^2 \Lambda_{n,s}|\b|\int_0^\infty r^{-1}\eta_{\ep}^2(r)dr.
\end{aligned}
\end{equation}
We compute the right-hand side of  the stability inequality \eqref{stability} for the test function $\phi_\alpha(x)=c_\alpha r^{-\frac{n}{2}+s}\eta_{\ep}(r)$ and $f_\alpha=\psi_\alpha(\theta) + (2\alpha-Q-1)s \log r$,
\begin{equation}
\label{3.exp-2}
\begin{aligned}
\int_{\mathbb R^n}  e^{-(f_\alpha-f_{\alpha-1})} (\phi_\alpha - \phi_{\alpha-1})^2
=& \int_0^\infty\int_{\b} \eta_{\ep}^2(r) r^{-2s}r^{-(n-2s)} r^{n-1} e^{-(\psi_\alpha-\psi_{\alpha-1})} (c_\alpha - c_{\alpha-1})^2 drd\theta\\
=& \int_0^\infty r^{-1}\eta^2_{\ep}(r)dr \int_{\b} e^{-(\psi_\alpha-\psi_{\alpha-1})} (c_\alpha - c_{\alpha-1})^2 d\theta.
\end{aligned}
\end{equation}
From the definition of the function $\eta_{\ep}$, we have
\begin{equation*}
\int_0^\infty r^{-1}\eta_{\ep}^2(r)dr=\log\frac{2}{\ep}+O(1).
\end{equation*}
One can see that both the first term on the right-hand side of \eqref{3.exp-1} and the right-hand side of \eqref{3.exp-2} carry the term $\int_0^\infty r^{-1}\eta_{\ep}^2(r)dr$ and it tends to $\infty$ as $\ep\to0$. Next, we claim that
\begin{equation}
\label{3.exp-3}
g_{\ep}(t):=\int_0^\infty r^{-1}\eta_{\ep}(r)(\eta_{\ep}(r)-\eta_{\ep}(rt))dr =O(\log t).
\end{equation}
From the definition of $\eta_{\ep},$ we have
\begin{equation*}
g_{\ep}(t)=\int_{\frac{\ep}{2}}^{\frac{2}{\ep}}r^{-1}\eta_{\ep}(r)(\eta_{\ep}(r)-\eta_{\ep}(rt))dr.
\end{equation*}
Notice that
\begin{equation*}
\eta_{\ep}(rt)=\begin{cases}
1, \quad &\mathrm{for}~\frac{\ep}{t}<r<\frac{1}{t\ep},\\
\\
0, \quad &\mathrm{for~either}~r<\frac{\ep}{2t}~\mathrm{or}~r>\frac{2}{t\ep}.
\end{cases}
\end{equation*}
Now we consider various ranges of value of $t\in(0,\infty)$ to establish the claim \eqref{3.exp-3}. Notice that
\begin{equation*}
g_{\ep}(t)\approx\begin{cases}
-\int_{\frac{\ep}{2}}^{\frac{2}{t\ep}}r^{-1}dr+\int_{\frac{\ep}{2}}^{\frac{2}{\ep}}r^{-1}dr\approx \log\ep=O(\log t),\quad~&\mathrm{if}~\frac{1}{t\ep}<\ep,\\
\\
-\int_{\frac{\ep}{2}}^{\ep}r^{-1}dr+\int_{\frac{1}{\ep t}}^{\frac{2}{\ep}}r^{-1}dr\approx \log t, \quad &\mathrm{if}~\frac{\ep}{t}<\ep<\frac{1}{\ep t},\\
\\
\int_{\frac{\ep}{2}}^{\frac{\ep}{t}}r^{-1}dr-\int_{\frac{1}{\ep}}^{\frac{2}{\ep }}r^{-1}dr\approx \log t,\quad & \mathrm{if}~\ep<\frac{\ep}{t}<\frac{1}{\ep},\\
\\
\int_{\frac{\ep}{2}}^{\frac{2}{\ep}}r^{-1}dr-\int_{\frac{\ep}{2t}}^{\frac{2\ep}{t}}r^{-1}dr\approx \log\ep=O(\log t),
\quad &\mathrm{if}~\frac{1}{\ep}<\frac{\ep}{t}.
\end{cases}
\end{equation*}
The other cases can be treated similarly. From this one can see that
\begin{equation*}
\begin{aligned}
&\int_0^{\infty}\left[\int_0^\infty r^{-1}\eta_{\ep}(r)(\eta_{\ep}(r)-\eta_{\ep}(rt))\right]\int_{\b}\int_{\b}
\dfrac{t^{n-1-\frac{n-2s}{2}}}
{(t^2+1-2t\langle\theta,\omega\rangle)^{\frac{n+2s}{2}}}d\omega d\theta dt\\
&\approx \int_0^\infty\int_{\b}\int_{\b}\dfrac{t^{n-1-\frac{n-2s}{2}}\log t}
{(t^2+1-2t\langle\theta,\omega\rangle)^{\frac{n+2s}{2}}}d\omega d\theta dt = O(1).
\end{aligned}
\end{equation*}
Collecting the higher order term $(\log\ep)$, we get
\begin{equation}
\label{Ans3}
\sum_{\alpha=1}^{Q-1} \int_{\b} e^{-(\psi_{\alpha+1}-\psi_{\alpha})} d\theta (c_{\alpha+1} - c_{\alpha})^2  \le \Lambda_{n,s} |\b| \sum_{\alpha=1}^Q c_\alpha ^2.
\end{equation}
Combining \eqref{integ} and \eqref{Ans3}, we conclude that
\begin{equation*}
2^{2s-1}\dfrac{\Gamma(\frac{n}{2})\Gamma(1+s)}{\Gamma(\frac{n-2s}{2})} \sum_{\alpha=1}^{Q-1}  \alpha(Q-\alpha) (c_{\alpha+1} - c_{\alpha})^2  \le 2^{2s}\frac{ \Gamma(\frac{n+2s}{4})^2  }{ \Gamma(\frac{n-2s}{4})^2}
 \sum_{\alpha=1}^Q c_\alpha ^2.
\end{equation*}
Applying the arguments in \cite{kwang}, this contradicts \eqref{assum3}. Therefore, such homogeneous solution does not exist and we finish the proof.	
\end{proof}

\section{Integral Estimates}\label{secen}
In this section we use the notion of stability to derive  energy estimates on $V_\alpha=e^{-(f_{\alpha+1}-f_\alpha)}$ and  to develop the integral representation formula of $f_\alpha$.
\begin{lemma}
\label{le4.1}
Let $\bff=(f_1,\cdots,f_Q)$ be a solution to \eqref{main} for some $n>2s$. Suppose that $\bff$ is stable outside a compact set of $\rn$. Then
\begin{equation}
\label{4.1}
\int_{B_r}V_\alpha dx\leq Cr^{n-2s},\quad \forall r\geq1,~\alpha=1,\cdots,Q.
\end{equation}
\end{lemma}

\begin{proof}
Let $R\gg1$ be a number such that $u$ is stable on $\rn\setminus B_R$. We define two smooth cut-off functions $\eta_R$ and $\varphi$ in $\rn$ such that
\begin{align*}
\eta_R(x)=\left\{\begin{array}{ll}
0\quad&\text{for }|x|\leq R\\
\\
1\quad&\text{for }|x|\geq 2R
\end{array}\right. ,\quad
\varphi(x)=\left\{\begin{array}{ll}
1\quad&\text{for }|x|\leq 1\\
\\
0\quad&\text{for }|x|\geq 2  \end{array}\right..
\end{align*}
Setting $\phi_\alpha=-\phi_{\alpha+1}=\eta_R(x)\varphi(\frac{x}{r})$ with $r\geq1$ and other $\phi_\beta$ to be zero if $\beta\neq\alpha,\alpha+1$. It is easy to see that
$$\sum_{\beta=1}^Q\phi_\beta=0.$$
Then $(\phi_1,\cdots,\phi_Q)$ is a good test function for the stability condition \eqref{stability}. Hence,
\begin{equation*}
\int_{B_r}V_\alpha dx\leq C+\int_{\rn}\left(|\s\phi_\alpha|^2+|\s\phi_{\alpha+1}|^2\right)dx
\leq C+Cr^{n-2s}\leq Cr^{n-2s},
\end{equation*}
where $n>2s$ and $r\geq1$ is used.
\end{proof}
As a direct consequence of the above lemma, we have the following estimate.
\begin{corollary}
\label{cr4.1}
Suppose $n>2s$ and $\bff=(f_1,\cdots,f_Q)$ is a solution of \eqref{main} which is stable outside a compact set. Then there exists $C>0$ such that
\begin{equation}
\label{4.2}
\int_{B_r}V_\alpha^\l\leq Cr^{n-2s}\quad \forall \l\geq1,~r\geq1,~\alpha=1,\cdots,Q , 
\end{equation}
where $V_\alpha^\l=e^{-(\fb^\l-\f^\l)}$.
\end{corollary}

Considering above decay estimate on $V_\alpha^\l$, together with the arguments applied in \cite[Lemma 2.3]{hy},  we conclude the following estimate.
\begin{lemma}
\label{le4.2}
For $\delta>0$ there exists $C=C(\delta)>0$ such that
$$\int_{\rn}\frac{V_\alpha^\l}{1+|x|^{n-2s+\delta}}dx\leq C,\quad \forall \l\geq 1,~\alpha=1,\cdots,Q.$$
\end{lemma}
For any $\alpha=1,\cdots,Q-1$, we set
\begin{equation*}
u_\alpha^\lambda :=f_\alpha^\lambda - f_{\alpha+1}^\lambda.
\end{equation*}
From equation \eqref{main},  we have
\begin{equation}
\label{4.3}
\s u_\alpha^\l=2e^{\f^\l-\fb^\l}-e^{\fa^\l-\f^\l}-e^{\fb^\l-f_{\alpha+2}^\l}
=2V_\alpha^\l-V_{\alpha-1}^\l-V_{\alpha+1}^\l,
\end{equation}
where $V_{0}^\l\equiv V_{Q}^\l\equiv 0$. Now, let
\begin{equation}
\label{4.4}
v_\alpha^\l:= c(n,s)\int_{\rn}\left(\frac{1}{|x-z|^{n-2s}}-\frac{1}{(1+|z|)^{n-2s}}\right)
\left(2V_\alpha^\l(z)-V_{\alpha-1}^\l(z)-V_{\alpha+1}^\l(z)\right)dz,
\end{equation}
where $c(n,s)$ is chosen such that
$$c(n,s)\s\frac{1}{|x-z|^{n-2s}}=\delta(x-z).$$
It is straightforward to notice that $v_{\alpha}^\l\in L_{\mathrm{loc}}^1(\rn)$. In what follows we show that $v_{\alpha}^\l\in L_s(\rn)$.

\begin{lemma}
\label{le4.3}
We have
\begin{equation}
\label{4.5}
\int_{\rn}\frac{|v_\alpha^\l|}{1+|x|^{n+2s}}dx\leq C,\quad \forall\l\geq1.
\end{equation}
\end{lemma}

\begin{proof}
Let
\begin{equation}
\label{4.6}
w_\alpha^\l(x):=c(n,s)\int_{\rn}\left(\frac{1}{|x-z|^{n-2s}}-\frac{1}{(1+|z|)^{n-2s}}\right)V_\alpha^\l(z)dz.
\end{equation}
From this and the definition of $v_\alpha^\l$, we have
$$v_\alpha^\l=2w_{\alpha}^\l-w_{\alpha-1}^\l-w_{\alpha+1}^\l.$$
On the other hand, applying the arguments of \cite[Lemma 2.4]{hy} one can conclude  that
\begin{equation*}
\int_{\rn}\frac{|w_{\alpha}^\l|}{1+|x|^{n+2s}}dx\leq C,\quad \forall \l\geq1,~\alpha=1,\cdots,Q.
\end{equation*}
As a direct consequence of this, we get
\begin{equation*}
\int_{\rn}\frac{|v_\alpha^\l|}{1+|x|^{n+2s}}dx\leq \int_{\rn}\frac{2|w_{\alpha}^\l|+|w_{\alpha-1}^\l|+|w_{\alpha+1}^\l|}{1+|x|^{n+2s}}dx\leq C.
\end{equation*}
This  finishes the proof.
\end{proof}
We now derive the following representation formula of $u_\alpha^\l$ that is an essential estimate in this section.
\begin{lemma}
\label{le4.4}
For any $\alpha=1,\cdots,Q-1,$ there exists constant $d_\alpha^\l\in\mathbb{R}$ such that
\begin{equation}
\label{4.7}
u_{\alpha}^\l(x)=c(n,s)\int_{\rn}\left(\frac{1}{|x-z|^{n-2s}}-\frac{1}{(1+|z|)^{n-2s}}\right)\left(2V_\alpha^\l(z)-V_{\alpha-1}^\l(z)-V_{\alpha+1}^\l(z)\right)dz+d_\alpha^\l.
\end{equation}
In addition, $d_\alpha^\l$ is bounded above for any $\l\geq1$ and $\alpha=1,\cdots,Q-1.$
\end{lemma}

\begin{proof}
According to the definition of $v_\alpha^\l$, one can easily see that
$$h_\alpha^\l:=u_\alpha^\l-v_\alpha^\l , $$
is a $s$-harmonic function in $\rn$.	Following ideas in \cite[Lemma 2.4]{h}, applied also in \cite{fwy,hy}, one can show that $h_\alpha^\l$ is either a constant, or a polynomial of degree one.

Next, we shall prove that the difference function $h_\alpha^\l$ must be constant. First, we give an estimation for the term $\int_{B_{2r}\setminus B_r}|w_\alpha^\l|dx$ in the ball $B_r$ with $r$ very large. By \eqref{4.6}, for $|x|$ large enough,  we have
\begin{equation}
\label{4.8}
\begin{aligned}
|w_\alpha^\l(x)|\leq~&C\int_{|z|\geq |x|^2}\left|\frac{1}{|x-z|^{n-2s}}-\frac{1}{(1+|z|)^{n-2s}}\right|V_\alpha^\l(z)dz +C\int_{|z-x|\leq \frac{|x|}{2}}\frac{1}{|x-z|^{n-2s}}V_\alpha^\l (z)dz\\
&+C\int_{|z|\leq |x|^2} \frac{1}{(1+|z|)^{n-2s}}V_\alpha^\l(z)dz,
\end{aligned}
\end{equation}
where we used $|x-z|\geq C|z|$ when $|z-x|\geq\frac{|x|}{2}$. Using Lemma \ref{le4.2},  we obtain
\begin{equation}
\label{4.9}
\begin{aligned}
\int_{|z|\geq |x|^2}\left|\frac{1}{|x-z|^{n-2s}}-\frac{1}{(1+|z|)^{n-2s}}\right|V_\alpha^\l(z)dz\leq~& C\int_{|z|\geq |x|^2}\frac{|x|}{1+|z|^{n-2s+1}}V_\alpha^\l(z) dz\\
\leq~& C\int_{|z|\geq |x|^2}\frac{|x|}{|z|^{\frac12+\delta}}
\frac{1}{1+|z|^{n-2s+\frac12-\delta}}V_\alpha^\l(z) dz\\
\leq~& C\int_{|z|\geq |x|^2}\frac{1}{1+|z|^{n-2s+\frac12-\delta}}V_\alpha^\l(z) dz\leq C,
\end{aligned}
\end{equation}
where $\delta$ is a small positive constant. While for the third term in the right-hand side of \eqref{4.8}, in the light of equation \eqref{4.2} and Lemma \ref{le4.2},  we have
\begin{equation}
\label{4.10}
\begin{aligned}
\int_{|z|\leq |x|^2} \frac{1}{(1+|z|)^{n-2s}}V_\alpha^\l(z)dz~&\leq C+C\left(1+\int_{2\leq|z|\leq|x|^2}\frac{1}{(1+|z|)^{n-2s}}V_\alpha^\l(z)dz\right)\\
~&\leq C+C\sum_{i=1}^{2[\log_2|x|]}\int_{2^i\leq |z|\leq 2^{i+1}}\frac{1}{(1+|z|)^{n-2s}}V_\alpha^\l(z)dz\\
~&\leq C+2C\log|x|.
\end{aligned}
\end{equation}
Combining \eqref{4.8}-\eqref{4.10}, we get
\begin{equation}
\label{4.11}
|w_\alpha^\l(x)|\leq C+C\log|x|+C\int_{|z-x|\leq \frac{|x|}{2}}\frac{1}{|x-z|^{n-2s}}V_\alpha^\l (z)dz\quad\mbox{for}~x~\mbox{large}.
\end{equation}
Then,
\begin{equation}
\label{4.12}
\begin{aligned}
\int_{B_{2r}\setminus B_r}|v_\alpha^\l(x)|dx\leq~&C\int_{B_{2r}\setminus B_r}(|w_\alpha^\l(x)|+|w_{\alpha-1}^\l(x)|+|w_{\alpha+1}^\l(x)|)dx\\
\leq~& C\int_{B_{2r}\setminus B_r}(1+\log(|x|+1))dx
+C\int_{r\leq|x|\leq 2r}\int_{|z-x|\leq\frac{|x|}{2}}\frac{1}{|x-z|^{n-2s}}V_\alpha^\l(z)dzdx\\
\leq~& Cr^n\log r,
\end{aligned}
\end{equation}
where $w_0^\l(x)\equiv w_Q^\l(x)=0.$ As a consequence, for $x\in B_{2r}\setminus B_r$,  we have
\begin{equation}
\label{4.13}
\left|\{x\mid r^{\frac12}\leq|w_\alpha^\l(x)|\}\cap\{x\mid r\leq|x|\leq 2r\}\right|\leq Cr^{n-\frac12}\log r.
\end{equation}
If $h_\alpha^\l$ is a polynomial of degree one, then we can find a subset $E_{r}$ of $B_{2r}\setminus B_r$ with measure greater than $Cr^n$ such that
$$u_\alpha^\l(x)\geq h_\alpha^\l-2|w_\alpha^\l|-|w_{\alpha-1}^\l|-|w_{\alpha+1}^\l|\geq Cr.$$
Therefore,
\begin{equation}
\label{4.14}
\int_{B_{2r}\setminus B_r}e^{u_\alpha^\l}dx\geq Cr^{n},
\end{equation}
which is in contradiction with Corollary \eqref{cr4.1}. Thus, $h_\alpha^\l$ must be constant, denoted by $d_\alpha^\l.$ Furthermore, using the fact that $v_\alpha^\l\in L_{\mathrm{loc}}^1(\rn)$ and applying Corollary \eqref{cr4.1},  one can show that $d_\alpha^\l$ is bounded above for any $\l\geq 1.$ This completes the proof.
\end{proof}

Now we prove a higher-order integrability of the nonlinearity of $V_\alpha$ on the region where $\bff=(f_1,\cdots,f_Q)$ is stable. More precisely, applying the stability inequality with appropriate test functions and using the $s$-extension arguments,  we establish the following $L^p$-estimates. The arguments are motivated by the ones established in \cite{df,CR} for the classical Liouville equation and in \cite{kwang} for the classical Toda system.

\begin{proposition}
\label{pr4.1}	
Let $f_\alpha\in C^{2\gamma}(\rn)\cap L^1(\rn,(1+|x|^{n+2s})dx)$	be a solution to \eqref{main}. Assume that $\bff$ is stable in $\rn\setminus B_R$ for some $R>0$. Then,  for every $p\in[1,\min\{5,1+\frac{n}{2s}\})$ there exists
$C=C(p)>0$ such that for $r$ large
\begin{equation}
\label{4.15}
\int_{B_{2r}\setminus B_r}e^{-p(f_{\alpha+1}-f_{\alpha})}dx\leq Cr^{n-2ps}.
\end{equation}
In particular,
\begin{itemize}
\item[(i)] for $|x|$ large,
\begin{align}
\label{4.16}
\int_{B_{|x|/2}(x)}e^{-p(f_{\alpha+1}(z)-f_{\alpha}(z))}dz
\leq C(p)|x|^{n-2ps}, \quad\forall p\in[1,\min\{5,1+\frac{n}{2s}\}),
\end{align}
\item[(ii)] for $r$ large
\begin{align}
\label{4.17}
\int_{B_r\setminus B_{2R}}e^{-p(f_{\alpha+1}(y)-f_{\alpha}(y))}dx\leq C(p)r^{n-2ps}, \quad\forall p\in[1,\min\{5, \frac{n}{2s}\}).
\end{align}
\end{itemize}	
\end{proposition}
The rest of this section is devoted to the proof of Proposition \ref{pr4.1}. Note that if $\bff$ is stable in $\Omega$ then
\begin{equation}
\label{4.18}
\sum_{\alpha}\int_{\r}y^{1-2s}|\nabla\Phi_\alpha|^2dxdy\geq
\kappa_s\sum_{\alpha}\int_{\rn}e^{-(f_\alpha-f_{\alpha-1})}(\phi_\alpha-\phi_{\alpha-1})^2dx,
\end{equation}
for every $\Phi_\alpha\in C_c^\infty(\overline{\r})$ satisfying $\phi_\alpha(\cdot):=\Phi_\alpha(\cdot,0)\in C_c^\infty(\Omega)$ for any $\alpha=1,\cdots,Q.$ Indeed, if $\overline{\phi}_\alpha$ is the $s$-harmonic extension of $\phi_\alpha$, we have
\begin{equation}
\label{4.19}
\begin{aligned}
\sum_\alpha\int_{\r}y^{1-2s}|\nabla\Phi_\alpha|^2dxdy
\geq~&\sum_\alpha\int_{\r}y^{1-2s}|\nabla\overline{\phi}_\alpha|^2dxdy\\
=~&\kappa_s\sum_\alpha\int_{\rn}\phi_\alpha\s\phi_\alpha dx\\
\geq~&\kappa_s\sum_{\alpha}\int_{\rn}e^{-(f_\alpha-f_{\alpha-1})}(\phi_\alpha-\phi_{\alpha-1})^2dx.
\end{aligned}
\end{equation}

In order to derive the higher-order integrability of $V_\alpha$, we need to study the $s$-extension of $V_\alpha$. In the light of  \eqref{maine}, one can show that the $s$-extension $\ov$ verifies the following equation,
\begin{equation}
\label{mainev}
\left\{
\begin{aligned}
\nabla\cdot(y^{1-2s}\nabla \ov)&=y^{1-2s}e^{-(\ofb-\of)}|\nabla(\of-\ofb)|^2&\quad&\text{in $\R^{n+1}_{+}$,}\\
  \ov&=  e^{-(\fb-\f)}&\quad&\text{on $\br$,}\\
-\lim_{y\to0} y^{1-2s}\partial_{y} \ov&= \kappa_sV_\alpha(\s\f-\s\fb)  &\quad&\text{on $\br$.}
\end{aligned}
\right.
\end{equation}

We use the above equation to establish the following lemma that is crucial for the proof of Proposition \ref{pr4.1}.

\begin{lemma}
\label{le4.5}
Let $f_\alpha\in C^{2\gamma}(\rn)\cap L^1(\rn,(1+|x|^{n+2s})dx)$ be a solution to \eqref{main}. Assume that $\bff$ is stable in $\Omega\subset\rn.$ Let $\Phi\in C_c^\infty(\overline{\r})$ be of the form $\Phi(x,y)=\varphi(x)\eta(y)$ for some $\varphi\in C_c^\infty(\Omega)$ and $\eta=1$ for $y\in[0,1]$. Then,  for every $0<q<2$ we have
\begin{equation}
\label{4.21}
\int_{\rn}V_\alpha^{2q+1}\varphi^2dx\leq C\int_{\r}y^{1-2s}\ov^{2q}|\nabla\Phi|^2dxdy+C\left|\int_{\r}\ov^{2q}\nabla\cdot(y^{1-2s}\nabla\Phi^2)dxdy\right|.
\end{equation}
\end{lemma}

\begin{proof}
Multiplying \eqref{mainev} by $\ov^{2q-1}\Phi^2$ and integration by parts leads to
\begin{equation}
\label{4.22}
\begin{aligned}
&(2q-1)\int_{\r}y^{1-2s}|\nabla\ov|^2\ov^{2q-2}\Phi^2dxdy+\int_{\r}y^{1-2s}\ov^{2q}|\nabla(\of-\ofb)|^2\Phi^2dxdy\\
&=\kappa_s\int_{\rn}V_\alpha^{2q}(2V_\alpha-V_{\alpha+1}-V_{\alpha-1})\varphi^2dx+\frac{1}{2q}\int_{\r}\ov^{2q}\nabla\cdot(y^{1-2s}\nabla\Phi^2)dxdy .
\end{aligned}
\end{equation}
We notice that $\nabla(\of-\ofb)=\frac{\nabla\ov}{\ov}$.  Therefore,  \eqref{4.22} can be written as
\begin{equation}
\label{4.23}
\begin{aligned}
2q\int_{\r}y^{1-2s}|\nabla\ov|^2\ov^{2q-2}\Phi^2dxdy=~&\kappa_s\int_{\rn}V_\alpha^{2q}(2V_\alpha-V_{\alpha+1}-V_{\alpha-1})\varphi^2dx\\
&+\frac{1}{2q}\int_{\r}\ov^{2q}\nabla\cdot(y^{1-2s}\nabla\Phi^2)dxdy.
\end{aligned}
\end{equation}
On the other hand, substituting $\phi_{\alpha+1}=-\phi_{\alpha}=V_\alpha^q\varphi$, $\phi_\beta=0$ if $\beta\neq\alpha,\alpha+1$ and $\Phi_\alpha=\phi_\alpha\eta$ into \eqref{4.18} we get
\begin{equation}
\label{4.24}
\begin{aligned}
2\kappa_s\int_{\rn}V_\alpha^{2q+1}\varphi^2dx\leq~& q^2\int_{\r}y^{1-2s}\ov^{2q-2}|\nabla\ov|^2\Phi^2dxdy+\int_{\r}y^{1-2s}\ov^{2q}|\nabla\Phi|^2dxdy\\
~&-\frac12\int_{\r}\ov^{2q}\nabla\cdot(y^{1-2s}\nabla\Phi^2)dxdy .
\end{aligned}
\end{equation}
Combining \eqref{4.23} and \eqref{4.24} for any $q\in[0,2)$ there exists a constant $C(q)$, depending only on $q$, such that for any $\varphi\in C_c^\infty(\Omega)$,
\begin{equation}
\label{4.25}
\int_{\rn}{{V_\alpha^2\varphi^2}}dx\leq C(q)\left(\int_{\r}y^{1-2s}\ov^{2q}|\nabla\Phi|^2dxdy+\left|\int_{\r}\ov^{2q}\nabla\cdot(y^{1-2s}\nabla\Phi^2)dxdy\right|\right).
\end{equation}
This completes the proof.
\end{proof}

Before we give the proof of Proposition \ref{pr4.1}, we present the following $L^1_{\mathrm{loc}}$-estimate for the $s$-extension function $ y^{1-2s}\ov^q$.

\begin{lemma}
\label{lem-2.7}
Let $V_\alpha^q\in L^1( \Omega) $ for some $\Omega\subset \B$. Then $y^{1-2s}\ov^q\in L^1_{\mathrm{loc}}(\Omega\times [0, \infty))$.
\end{lemma}

\begin{proof}
Let $\Omega_0\Subset\Omega$ be fixed.  Since $f_\alpha\in  L_s(\B)$, for $x\in\Omega_0$ and $y\in(0,R)$ we have
\begin{align*}\of(x,y)-\ofb(x,y)\leq~& C+\int_{\Omega}(\f(z)-\fb(z))P(X,z)dz\\
=~&C+\int_{\Omega}g(x,z)(\f(z)-\fb(z))\frac{P(X,z)dz}{g(x,z)},
\end{align*}
where
$1\geq  g(x,z):=\int_{\Omega}P(X,z)dz\geq C$ for some positive constant $C$ depending on $R$, $\Omega_0$ and $\Omega$ only. Therefore, by Jensen's inequality
\begin{align*}
\int_{\Omega_0}e^{q(\of(x,y)-\ofb(x,y))}dx
&\leq C \int_{\Omega_0}\int_{\Omega}e^{q g(x,z)f_\alpha(z)-f_{\alpha+1}(z)}P(X,z)dzdx\\
&\leq C\int_{\Omega}\max\{e^{q(\f-\fb)},1\}\int_{\Omega_0}P(X,z)dxdz
\\&\leq C+C\int_{\Omega}e^{q(\f(z)-\fb(z))}dz,
\end{align*}
where the constant $C$ depends on $R,~\Omega_0$ and $\Omega$, but not on $y$. Hence,
\begin{align*}
\int_{\Omega_0\times(0,R)} y^{1-2s}e^{q(\of(x,y)-\ofb(x,y))}dxdy \leq
\int_0^R y^{1-2s}\int_{\Omega_0}e^{q(\of(x,y)-\ofb(x,y))}dxdy <\infty.
\end{align*}
This finishes the proof.
\end{proof}

Now we are ready to prove Proposition \ref{pr4.1}.

\begin{proof}[Proof of Proposition \ref{pr4.1}.]
From the H\"older's inequality and Lemma \ref{le4.1},  for $p\in(0,1)$ we get
\begin{equation*}
\int_{B_{2r}\setminus B_r}V_\alpha^pdx\leq \left(\int_{B_{2r}\setminus B_r}V_\alpha dx \right)^p\left(\int_{B_{2r}\setminus B_r} 1dx\right)^{1-p}
\leq Cr^{n-2ps}.
\end{equation*}
Next, we claim that if
\begin{equation}
\label{4.26}
\int_{B_{2r}\setminus B_r}V_\alpha^pdx\leq Cr^{n-2ps}\quad\mathrm{for}~r>2R,
\end{equation}
for some $p\in(0,\min\{\frac{n}{2s},4\})$, then
\begin{equation}
\label{4.27}
\int_{B_{2r}\setminus B_r}V_\alpha^{p+1}dx\leq Cr^{n-2(1+p)s}\quad\mathrm{for}~r>3R.
\end{equation}
By \eqref{4.26}, it is straightforward  to show that
\begin{equation}
\label{4.28}
\int_{B_r\setminus B_{2R}}V_\alpha^pdx\leq Cr^{n-2ps}\quad \text{for  }r> 2R.
\end{equation}
Indeed, for  $r>2R$ of the form $r=2^{N_1}$ with some positive integer $N_1$,  and taking  $N_2$ to be the smallest integer such that $2^{N_2}\geq 2R$,  by  \eqref{4.26} we deduce
\begin{equation}
\label{4.29}
\begin{aligned}
\int_{B_{2^{N_1}}\setminus B_{2R}}V_\alpha^pdx=
&\int_{B_{2^{N_2}}\setminus B_{2R}}V_\alpha^p dx+\sum_{\ell=1}^{N_1-N_2}\int_{B_{2^{N_2+\ell}}\setminus B_{2^{N_1+\ell}}}V_\alpha^p dx\\
\leq & ~C+ C\sum_{\ell=1}^{N_1-N_2}(2^{N_2+\ell})^{n-2p s}
\leq C2^{N_1(n-2p s)},
\end{aligned}
\end{equation}
where we used $n-2p s>0.$ Then using the hypothesis \eqref{4.26},   we derive the following decay estimate
\begin{equation}
\label{4.30}
\begin{aligned}
\int_{|z|\geq r}\frac{V_\alpha^p}{|z|^{n+2s}}dz=
\sum_{i=0}^\infty\int_{2^{i+1}r\geq |z|\geq 2^ir}\frac{V_\alpha^p}{|z|^{n+2s}}dz
\leq \frac{ C}{r^{2s+2p s}}\sum_{i=0}^\infty\frac{1}{2^{(2+2p)si}}
\leq Cr^{-2s-2ps}.
\end{aligned}
\end{equation}
On the other hand, from the Poisson's formula for $|x|>\frac52R$ we have
\begin{align*}
(\of(x,y)-\ofb(x,y)) \leq~& C\frac{y^{2s}}{(R+y)^{n+2s}}\int_{|z|\leq 2R}(\f(z)-\fb(z))^+(z)dz\\
&+\int_{\B}\chi_{\B\setminus B_{2R}}(z)(\f(z)-\fb(z))P(X,z)dz\\
\leq ~&C+\int_{\B}\chi_{\B\setminus B_{2R}}(z)(\f(z)-\fb(z))P(X,z)dz,
\end{align*}
where $\chi_A$ denotes   the characteristic function of a set $A$. Using the Jensen's inequality, we obtain
\begin{align}
\label{4.31}
\ov^p(x,y)\leq C\int_{\B}\left(V_\alpha^{p}\chi_{\B\setminus B_{2R}}(z)+\chi_{B_{2R} }(z) \right)P(X,z)dz \quad\text{for }|x|\geq 3R.
\end{align}
 Integrating both sides of the inequality \eqref{4.31}, with respect to $x$, on $B_{r}\setminus B_{3R}$ for $r>3R$ and $y\in(0,r)$, we get
\begin{align*}
\int_{B_{r}\setminus B_{3R}}\ov^pdx
&\leq C \int_{|z|\leq 2R}\int_{|x|\leq r}P(X,z)dxdz+C\int_{|z|\geq 2r}\int_{|x|\leq r}V_\alpha^pP(X,z)dxdz\\
&\quad+C\int_{2R\leq |z|\leq 2r}V_\alpha^p(z)\int_{|x|\leq r}P(X,z)dxdz \\ &\leq CR^n +Cr^{n+2s}\int_{|z|\geq 2r}\frac{V_\alpha^p(z)}{|z|^{n+2s}}dz+C\int_{2R\leq |z|\leq 2r}V_\alpha^p(z)dz\\
&\leq C+Cr^{n-2ps}\leq Cr^{n-2ps},
\end{align*}
where we used \eqref{4.26}, \eqref{4.29} and \eqref{4.30}.  Now we fix non-negative smooth functions $\varphi$ and $\eta$ on $\B$ and $ [0,\infty)$, respectively, such that
\begin{align*}
\varphi(x)=\left\{\begin{array}{ll}
1 &\quad\text{in } B_{2}\setminus B_1\\
\\
0&\quad\text{in }B_{2/3}\cup B_{3}^c
\end{array}\right.,\quad
\eta(t)=\left\{\begin{array}{ll}
1 &\quad\text{in }[0,1]\\
\\
0&\quad\text{in } [2,\infty)
\end{array}\right..
\end{align*}
For $r>0$, we now set $\Phi_r(x,t)=\varphi(\frac xr)\eta(\frac tr)$. Applying the test function  $\Phi_r$ in Lemma \ref{le4.5} for $r\geq 3R$ and using the fact that  $|\nabla \Phi_r |\leq  \frac Cr$, we conclude
\begin{equation*}
\begin{aligned}
\int_{\r}y^{1-2s}\ov^p|\nabla  \Phi_r |^2dxdy
&\leq Cr^{-2}\int_0^{2r} y^{1-2s}\int_{B_{3r}\setminus B_{2r/3}}\ov^p(x,y)dxdy\\
&\leq Cr^{n-2-2p s}\int_0^{2r}y^{1-2s} dt\\
&\leq Cr^{n-2s(1+p)}.
\end{aligned}
\end{equation*}
With  similar arguments, we have
\begin{equation*}
\left|\int_{\R}\ov^p \nabla\cdot [y^{1-2s}\nabla  \Phi_r^2]dxdt\right|\leq Cr^{n-2s(1+p)}.
\end{equation*}
Therefore, \eqref{4.27} follows from \eqref{4.21} of Lemma \ref{le4.5} as desired. This completes the proof of the claim. Repeating the above arguments finitely many times we get \eqref{4.15} in the light of the fact $V_\alpha=e^{-(\fb-\f)}$.  The estimate \eqref{4.16} follows immediately as well since $B_{|x|/2}(x)\subset B_{2r}\setminus B_{r/2}$ with $r=|x|$. Similarly, from \eqref{4.28}  one can obtain  \eqref{4.17} of Proposition \ref{pr4.1}. This completes the proof.
\end{proof}

\section{Blow-down Analysis}\label{secbd}
In this section, we perform blow-down analysis for the energy functional given by the monotonicity formula \eqref{energyEs1}.  We start with the proof of the latter formula. 

\begin{proof}[Proof of Theorem \ref{thmono1s}] Let
\begin{equation} \label{I}
I(\overline{ \bff },\lambda) =  \lambda^{2s-n} \sum_{\alpha} \left(\frac12\int_{\r\cap B_\lambda^{n+1}} y^{1-2s} \vert\nabla \of\vert^2\;dx\,dy - \kappa_{s} \int_{\br\cap B_\lambda^{n+1}}  e^{-(\of - \ofa)}\;dx\right) . 
\end{equation}
Now for $X\in\r$, define
\begin{equation}\label{ulambdas1}
\of^\lambda(X) : = \of(\lambda X)  - (2\alpha - Q-1)s \log \lambda . 
\end{equation}
Then, $\of^\lambda$ solves \eqref{maine} and in addition
\begin{equation}\label{t3}
I(\overline{ \bff },\lambda) = I(\overline{ \bff }^\lambda,1).
\end{equation}
Taking partial derivatives on $\partial B_1$, we get
\begin{equation} \label{lambdader}
 \partial_r \of^\lambda=\lambda \partial_\lambda \of ^\lambda +(2\alpha-Q-1)s.
\end{equation}
Differentiating the formula \eqref{I} with respect to $\lambda$, we find
$$
\partial_\lambda I(\overline{ \bff },\lambda) = \sum_{\alpha}\int_{\r\cap B_1^{n+1}}y^{1-2s}\nabla \of^\lambda\cdot\nabla\partial_\lambda \of^\lambda dx\,dy
+\kappa_{s}\sum_{\alpha}
\int_{\br\cap B_1^{n+1}}  e^{-(\of^\lambda - \ofa^\lambda)}( \partial_\lambda \of^\lambda -  \partial_\lambda \ofa^\lambda ) dx.
$$
Integrating by parts and then using \eqref{lambdader},
\begin{align*}
\partial_\lambda I(\overline{ \bff },\lambda)  &= \sum_\alpha\int_{\partial B_1^{n+1}\cap\r}y^{1-2s} \partial_r  \of^\lambda \partial_\lambda  \of^\lambda d\sigma
\\&= \sum_\alpha\left(\lambda \int_{\partial B_1^{n+1}\cap\r}y^{1-2s} (\partial_\lambda  \of^\lambda)^2 d\sigma +
(2\alpha-Q-1)s \int_{\partial B_1^{n+1}\cap\r}y^{1-2s}  \partial_\lambda  \of^\lambda  d\sigma\right)\\
&= \sum_\alpha\left[(\lambda \int_{\partial B_1^{n+1}\cap\r} y^{1-2s} (\partial_\lambda  \of^\lambda)^2 d\sigma +
\partial_\lambda \left( (2\alpha-Q-1)s\int_{\partial B_1^{n+1}\cap\r}y^{1-2s} \of^\lambda \;d\sigma\right)\right].
\end{align*}
This implies that
\begin{equation}
\partial_\lambda\left[ I( \overline{ \bff },\lambda)-\sum_\alpha(2\alpha-Q-1)s\int_{\partial B_1^{n+1}\cap\r}y^{1-2s}  \of^\lambda  \;d\sigma\right] = \lambda \int_{\partial B_1^{n+1}\cap\r} y^{1-2s} \sum_\alpha(\partial_\lambda   \of^\lambda)^2 d\sigma .
\end{equation}
Applying the scaling (\ref{ulambdas1}) completes the proof.
\end{proof}

We now analyze the third term in the monotonicity formula.
\begin{lemma}
\label{le5.1}
Let $\of^\l$ be the $s$-harmonic extension of $\f^\l$, then
\begin{equation}
\label{5.1}
\begin{aligned}
\sum_\alpha(Q+1-2\alpha)\int_{\partial B_1^{n+1}\cap\r}y^{1-2s}\of ^\l d\sigma=c_s\sum_{\ell=1}^{[\frac{Q}{2}]}(\ell Q-\ell^2)(d_\ell^\l+d_{Q-\ell}^\l)+O(1)	, 
\end{aligned}
\end{equation}
where $[\tau]$ denotes the integer part of $\tau$, $d_\ell^\l$ is defined in \eqref{4.7} and $c_s$ is a positive finite number given by
$$c_s:=\int_{\partial B_1^{n+1}\cap\r}y^{1-2s}d\sigma.$$
\end{lemma}

\begin{proof}
By a simple observation, we can rewrite $\sum_\alpha(Q+1-2\alpha)\of^\l$ as
\begin{equation}
\label{5.2}
\begin{aligned}
\sum_\alpha(Q+1-2\alpha)\of^\l=~&\sum_{\ell=1}^{[\frac{Q}{2}]}(\ell Q-\ell^2)\left[(\overline{f}_{\ell}^\l-\overline{f}_{\ell+1}^\l)+(\overline{f}_{Q-\ell}^\l-\overline{f}_{Q-\ell+1}^\l)\right]\\
=~&\sum_{\ell=1}^{[\frac{Q}{2}]}(\ell Q-\ell^2)(\overline{u}_\ell^\l+\overline{u}_{Q-\ell}^\l),
\end{aligned}
\end{equation}
where $\ou_\alpha^\l$ is the $s$-extension of $u_\alpha^\l$. In order to derive \eqref{5.1}, it suffices to show that
\begin{equation}
\label{5.3}
\int_{\partial B_1^{n+1}\cap\r}y^{1-2s}\ou_\alpha^\l d\sigma=c_sd_\alpha^\l+O(1),\quad \forall\alpha=1,\cdots,Q-1.
\end{equation}
From \eqref{4.7} and using the Poisson formula we have
\begin{equation*}
\ou_\alpha^\l(X)=\int_{\B}P(X,z)u_\alpha^\l(z)dz=d_\alpha^\l+\int_{\B}P(X,z)v_\alpha^\l(z)dz,
\end{equation*}
where $v_\alpha^\l$ is given in \eqref{4.4}. Based on the expression formula of the Poisson kernel and Lemma \ref{le4.3}, one can get that
\begin{equation*}
\int_{\B\setminus B_2}P(X,z)|v_\alpha^\l(z)|dz\leq C\quad\mbox{for}\quad |X|\leq 1.
\end{equation*}
Therefore
\begin{align*}
\int_{\partial B^{n+1}_1\cap\r}y^{1-2s}\ou_\alpha^\l(X)d\sigma
=c_sd_\alpha^\l+O(1)+\int_{\partial B_1^{n+1}\cap\r}y^{1-2s}\int_{B_2}P(X,z)v_\alpha^\l(z)dzd\sigma.
\end{align*}
We denote the last term in the above equation by $II$. To estimate the term $II$, we claim that
\begin{equation}
\label{5.4}
\int_{\partial B^{n+1}_1\cap\r}\int_{B_4}y^{1-2s}P(X,z)\frac{1}{|\xi-z|^{n-2s}}dzd\sigma\leq C\quad \mbox{for every}~\xi\in\B.
\end{equation}
Indeed, for $x\neq \xi$ we set $r_{x,\xi}=\frac12|x-\xi|$. Then we have
\begin{equation}
\label{5.5}
\begin{aligned}
&\int_{B_4}\frac{y}{|(x-z,y)|^{n+2s}|\xi-z|^{n-2s}}dz\\
&\leq\left(\int_{B(\xi,r_{x,\xi})}+\int_{B_4\setminus  B(\xi,r_{x,\xi})}\right)\frac{y}{|(x-z,y)|^{n+2s}|\xi-z|^{n-2s}}dz\\
&\leq C\frac{y}{(r_{x,\xi}+y)^{n+2s}}\int_{B(\xi,r_{x,\xi})}\frac{1}{|\xi-z|^{n-2s}}dz
+\frac{1}{r_{x,\xi}^{n-2s}}\int_{B_4\setminus  B(\xi,r_{x,\xi})}\frac{y}{|(x-z,y)|^{n+2s}}dz\\
&\leq C\left(\frac{yr_{x,\xi}^{2s}}{(r_{x,\xi}+y)^{n+2s}}
+\frac{y^{1-2s}}{r_{x,\xi}^{n-2s}}\right)\leq C\left(\frac{1}{r_{x,\xi}^{n-1}}+\frac{y^{1-2s}}{r_{x,\xi}^{n-2s}}\right).
\end{aligned}
\end{equation}
We use the stereo-graphic projection $(x,y)\to\theta$ from $\partial B_1^{n+1}\cap\r\to\mathbb{R}^n\setminus B_1$, i.e.,
$$(x,y)\to\theta=\frac{x}{1-y}.$$
Then $r_{x,\xi}=\frac12\left|\frac{2\theta}{1+|\theta|^2}-\xi\right|$ and it follows that
\begin{equation*}
\int_{\partial B_1^{n+1}\cap\r}\frac{1}{r_{x,\xi}^{n-1}}d\sigma
\leq\int_{|\theta|\geq 1}\frac{1}{r_{x,\xi}^{n-1}}\frac{1}{(1+|\theta|^2)^n}d\theta\leq C,
\end{equation*}
and
\begin{equation}
\label{5.6}
\begin{aligned}
\int_{\partial B_1^{n+1}\cap\r}\frac{y^{1-2s}}{r_{x,\xi}^{n-2s}}d\sigma
\leq\int_{|\theta|\geq1}\frac{y^{1-2s}}{r_{x,\xi}^{n-2s}}\frac{1}{(1+|\theta|^2)^n}d\theta\leq C.
\end{aligned}
\end{equation}
From \eqref{5.5}-\eqref{5.6}, we proved \eqref{5.4}. As a consequence, we have
\begin{equation*}
\begin{aligned}
|II|&\leq C+C\int_{\partial B^{n+1}_1\cap\r}y^{1-2s}\int_{|z|\leq 2}P(X,z)\int_{|\xi|\leq 4}\frac{2V_\alpha^\l(\xi)+V_{\alpha_1}^\l(\xi)+V_{\alpha+1}^\l(\xi)}{|z-\xi|^{n-2s}}d\xi dzd\sigma\\
&\leq C+C\int_{|(\xi)|\leq 4}e^{u^\l((\xi))}\int_{\partial B_{1}^{n+1}\cap\r}\int_{|z|\leq 2}y^{1-2s}P(X,z)\frac{1}{|z-\xi|^{n-2s}}dzd\sigma d\xi\\
&\leq C+C\int_{|\xi|\leq 4}\left(2V_\alpha^\l(\xi)+V_{\alpha_1}^\l(\xi)+V_{\alpha+1}^\l(\xi)\right)d\xi\leq C,
\end{aligned}
\end{equation*}
where we used \eqref{4.1} and \eqref{4.2}. Hence, we finish the proof.
\end{proof}

We use the arguments of Lemma \ref{le5.1} to derive the following weighted $L^2$-estimate of the $s$-extension of $v_\alpha^\l$, see \eqref{4.4} for the definition of $v_\alpha^\l$.  
\begin{lemma}
\label{le33.12}
Let $\overline{v}_\alpha^\l$ denote the extension of $v_\alpha^\l$, then for $r\geq1$ we have
\begin{equation}
\label{eq.33-1}
\int_{B_r^{n+1}\cap\r}y^{3-2s}|\overline{v}_\alpha^\l |^2dxdy \leq C(r),
\end{equation}
for some constant $C(r)$ depending only on $r.$
\end{lemma}

\begin{proof}
From the proof of Lemma \ref{le5.1}, it is enough  to show that
\begin{equation}
\label{eq.33-2}
\int_{B_r^{n+1}\cap\r}y^{1-2s}\left(\int_{|z|\leq2r}P(X,z)\int_{|\xi|\leq 4r}\frac{V_\alpha^\l(\xi)+V_{\alpha-1}^\l(\xi)+V_{\alpha+1}^\l(\xi)}{|\xi-z|^{n-2s}}d\xi dz\right)^2dxdy \leq C(r).
\end{equation}
For $x\neq \zeta$, following the arguments in \eqref{5.5} we have
\begin{equation}
\label{eq.33-3}
\begin{aligned}
\int_{|z|\leq 2r}\frac{y^{2s}}{|(x-z,y)|^{n+2s}|\zeta-z|^{n-2s}}dz\leq \frac{C}{|x-\zeta|^{n-2s}}.
\end{aligned}
\end{equation}
As a consequence, we get
\begin{equation}
\label{eq.33-4}
\int_{B_r^{n+1}\cap\r}y^{1-2s}|\overline{v}_\alpha^\l|^2dxdy\leq C+C\int_{B_r^{n+1}\cap\r}y^{1-2s}\left(\int_{|\xi|\leq 4r}\frac{V_\alpha^\l(\xi)+V_{\alpha-1}^\l(\xi)+V_{\alpha+1}^\l(\xi)}{|x-\xi|^{n-2s}}d\xi\right)^2dxdy.
\end{equation}
For $x\neq0$, we write
\begin{equation}
\label{eq.33-5}
\int_{|\xi|\leq 4r}\frac{V_\alpha^\l(\xi)}{|x-\xi|^{n-2s}}d\xi
=\left(\int_{B_{2\l^{-1}R}}+\int_{(B_{4r}\setminus B_{2\l^{-1}R})\setminus B(x,\frac{|x|}{2})}+\int_{(B_{4r}\setminus B_{2\l^{-1}R})\cap B(x,\frac{|x|}{2})}\right)\frac{V_\alpha^\l(\xi)}{|x-\xi|^{n-2s}}d\xi.
\end{equation}
Here $R$ is chosen such that $\bff$ is stable in $\rn\setminus B_R.$ For the first term on the right-hand side, using $f_\alpha\in C^{2\gamma}(\rn)$ we have
\begin{equation}
\label{eq.33-6}
\int_{B_{2\l^{-1}R}}\frac{V_\alpha^\l(\xi)}{|x-\xi|^{n-2s}}d\xi\leq C(\l^{-1}R)^{2s}\l^{2s}\leq C.
\end{equation}
Fo the second term on the right-hand side,
\begin{equation}
\label{eq.33-7}
\begin{aligned}
\int_{(B_{4r}\setminus B_{2\l^{-1}R})\setminus B(x,\frac{|x|}{2})}\frac{V_\alpha^\l(\xi)}{|x-\xi|^{n-2s}}d\xi\leq~&
C+\int_{B_{4r}\setminus B_{2|x|}}\frac{V_\alpha^\l(\xi)}{|\xi|^{n-2s}}d\xi
\\\leq~& C+\sum_{i=1}^{[|\log_2(\frac{2r}{|x|})|]}\int_{B_{2^{i}|x|}\setminus B_{2^{i-1}|x|}}\frac{V_\alpha^\l(\xi)}{|\xi|^{n-2s}}d\xi\\
\leq~& C(r)+C(r)|\log|x||.
\end{aligned}
\end{equation}
For the third term, we have
\begin{equation}
\label{eq.33-7}
\begin{aligned}
\left(\int_{(B_{4r}\setminus B_{2\l^{-1}R})\cap B(x,\frac{|x|}{2})} \frac{V_\alpha^\l(\xi)}{|x-\xi|^{n-2s}}d\xi\right)^2\leq~& \left(\int_{B(x,\frac{|x|}{2})} \frac{V_\alpha^\l(\xi)}{|x-\xi|^{n-2s}}d\xi\right)^2\\
\leq~&\int_{B(x,\frac{|x|}{2})} \frac{1}{|x-\xi|^{n-2s}}d\xi\int_{B(x,\frac{|x|}{2})} \frac{(V_\alpha^\l(\xi))^2}{|x-\xi|^{n-2s}}d\xi.
\end{aligned}
\end{equation}
From \eqref{eq.33-4}-\eqref{eq.33-7} and Proposition \ref{pr4.1}-(i), we get
\begin{equation}
\label{eq.33-8}
\begin{aligned}
\int_{B_{r}^{n+1}\cap\r}y^{1-2s}|\overline{v}_\alpha^\l|^2dxdy
\leq~&C(r)+\int_{B_{r}^{n+1}\cap\r}y^{1-2s}\left(C(r)+C(r)|\log|x||^2\right)dxdy\\
&+\int_{B_{r}^{n+1}\cap\r}y^{1-2s}|x|^{2s}
\int_{B(x,\frac{|x|}{2})}\frac{(V_\alpha(\xi)+V_{\alpha+1}(\xi)+V_{\alpha-1}(\xi))^2}{|x-\xi|^{n-2s}}d\xi dxdy\\
\leq~&C(r)+C(r)\int_{(B_{r}^{n+1}\cap\r)\cap\{x\mid|x|\leq 4\l^{-1}R\}}y^{1-2s}\frac{|x|^{2s}}{|x-\xi|^{n-2s}}\l^{4s}dxdy\\
&+C(r)\int_{B_r\setminus B_{4\l^{-1}R}}\frac{|x|^{2s}}{|x-\xi|^{n-2s}}\int_{B(x,\frac{|x|}{2})}(V_\alpha(\xi)+V_{\alpha+1}(\xi)+V_{\alpha-1}(\xi))^2d\xi dx\\
\leq~&C(r),
\end{aligned}
\end{equation}
where we used
$$\int_{B(x,\frac{|x|}{2})}(V_\alpha(\xi)+V_{\alpha+1}(\xi)+V_{\alpha-1}(\xi))^2d\xi\leq C\l^{4s}\quad\mbox{for}~|x|\leq 4\l^{-1}R.$$
This finishes the proof.
\end{proof}

In order to estimate the first quadratic term in the monotonicity formula \eqref{energyEs1},  we need the following result.

\begin{lemma}
\label{le5.2}
We have
$$\int_{B_r}|(-\D)^\frac s2u_\alpha^\lambda(x)|^2dx\leq Cr^{n-2s}\quad\text{for~ }r\geq1,~\ \lambda\geq1,\quad \alpha=1,\cdots,Q-1.$$
\end{lemma}
\begin{proof}
By a scaling argument,  it suffices to prove the lemma for $\lambda=1$. It follows from \eqref{4.7} that
$$(-\D)^\frac s2u_\alpha(x)=C\int_{\B}\frac{1}{|x-z|^{n-s}}(2V_\alpha(z)-V_{\alpha-1}(z)-V_{\alpha+1}(z))dz,$$
in the sense of distribution. Let $R>0$ be such that $\bff$ is stable outside $B_R$. For $r\gg R$, we decompose $B_{r}=B_{4R}\cup (B_{r}\setminus B_{4R}).$  Since  $u_\alpha\in\dot H_{\mathrm{loc}}^s(\B)$ for any $\alpha=1,\cdots,Q-1$,  we get
\begin{equation}
\label{5.7}
\int_{B_{4R}}|(-\D)^\frac{s}{2} u_\alpha|^2dx\leq C(R).
\end{equation}
While for $4R<|x|<r$ we estimate
\begin{equation}
\label{5.8}
\begin{aligned}
|(-\D)^\frac s2u_\alpha(x)| &\leq C\left(\int_{B_{2R}}    +\int_{B_{2r}\setminus B_{2R}}+\int_{\B\setminus B_{2r}}\right)
\frac{1}{|x-z|^{n-s}}\left(2V_\alpha(z)+V_{\alpha-1}(z)+V_{\alpha+1}(z)\right)dz\\
&\leq  \frac{C}{|x|^{n-s}}+C\int_{B_{2r}\setminus B_{2R}}\frac{1}{|x-z|^{n-s}}\left(2V_\alpha(z)+V_{\alpha-1}(z)+V_{\alpha+1}(z)\right)dz
\\
&\quad + C\int_{\B\setminus B_{2r}}
\frac{1}{|z|^{n-s}}\left(2V_\alpha(z)+V_{\alpha-1}(z)+V_{\alpha+1}(z)\right)dz \\
&=:C\left(\frac{1}{|x|^{n-s}}+I_1(x)+I_2 \right) .
\end{aligned}
\end{equation}
Using \eqref{4.1},  we bound the last term in the above as
\begin{align*}
|I_2|=\sum_{k=0}^\infty \int_{2^kr\leq|x|\leq 2^{k+1}r}\left|\frac{2V_\alpha(z)+V_{\alpha-1}(z)+V_{\alpha+1}(z)}{|z|^{n-s}}\right|dz  \leq C\sum_{k=0}^\infty \frac{(2^{k+1}r)^{n-2s}}{(2^kr)^{n-s}}\leq \frac{C}{r^{s}}.
\end{align*}
Therefore,
\begin{equation}
\label{5.9}
\int_{B_r}I_2^2dx\leq Cr^{n-2s}.
\end{equation}
In regards to $I_1(x)$ in the above, we have
\begin{equation}
\label{5.10}
\begin{aligned}
I_1(x)\leq~&\int_{\{z|2R\leq|z|\leq 2|x|\}\cap B(x,\frac{|x|}{2})}
\frac{2V_\alpha(z)+V_{\alpha-1}(z)+V_{\alpha+1}(z)}{|x-z|^{n-s}}dz\\
&+\int_{\{z|2R\leq|z|\leq 2|x|\}\setminus B(x,\frac{|x|}{2})}
\frac{2V_\alpha(z)+V_{\alpha-1}(z)+V_{\alpha+1}(z)}{|x-z|^{n-s}}dz\\
&+\int_{2r\geq|z|\geq 2|x|}\frac{2V_\alpha(z)+V_{\alpha-1}(z)+V_{\alpha+1}(z)}{|x-z|^{n-s}}dz.
\end{aligned}
\end{equation}
For the second term on the right-hand side of \eqref{5.10}, we have
\begin{equation}
\label{5.11}
\begin{aligned}
\left|\int_{\{z|2R\leq|z|\leq 2|x|\}\setminus B(x,\frac{|x|}{2})}
\frac{2V_\alpha(z)+V_{\alpha-1}(z)+V_{\alpha+1}(z)}{|x-z|^{n-s}}dz\right|
\leq C\frac{|x|^{n-2s}}{|x|^{n-s}}\leq \frac{C}{|x|^s},
\end{aligned}
\end{equation}
while for the third term on the right-hand side of \eqref{5.10}, following the estimation of $I_2$ we get
\begin{equation}
\label{5.12}
\left|\int_{2r\geq|z|\geq 2|x|}\frac{2V_\alpha(z)+V_{\alpha-1}(z)+V_{\alpha+1}(z)}{|x-z|^{n-s}}dz\right|\leq\frac{C}{|x|^s}.
\end{equation}
Concerning the first term on the right-hand side of \eqref{5.10},  from \eqref{5.11}-\eqref{5.12} and H\"older's inequality, we obtain
\begin{equation}
\label{5.13}
I_1^2(x)\leq C\left( \frac{1}{|x|^{2s}}+
\int_{B(x,\frac{|x|}{2})}\frac{1}{|x-z|^{n-s}}dz
\int_{B(x,\frac{|x|}{2})}\frac{V_\alpha^2(z)+V_{\alpha-1}^2(z)+V_{\alpha+1}^2(z)}{|x-z|^{n-s}}dz
\right).
\end{equation}
Now, using Proposition \ref{pr4.1} we conclude
\begin{equation}
\label{5.14}
\begin{aligned}
\int_{B_r}I_1^2(x)\leq~&C\int_{B_r\setminus B_{4R}}\left(\frac{1}{|x|^{2s}}+
\int_{B(x,\frac{|x|}{2})}\frac{1}{|x-z|^{n-s}}dz
\int_{B(x,\frac{|x|}{2})}\frac{V_\alpha^2(z)+V_{\alpha-1}^2(z)+V_{\alpha+1}^2(z)}{|x-z|^{n-s}}dz
\right)dx\\
\leq~&Cr^{n-2s}+C\int_{B_r\setminus B_{4R}}\frac{|x|^s}{|x-z|^{n-s}}dx
\int_{B(x,\frac{|x|}{2})}(V_\alpha^2(z)+V_{\alpha-1}^2(z)+V_{\alpha+1}^2(z))dz\\
\leq~&Cr^{n-2s}+C\sum_{i=2}^{[\log_2\left(\frac{r}{4R}\right)]+1}\left(\int_{B_{2^{i+1}R}\setminus B_{2^iR}}
\frac{|x|^s}{|x-z|^{n-s}}dx
\int_{B(x,\frac{|x|}{2})}(V_\alpha^2(z)+V_{\alpha-1}^2(z)+V_{\alpha+1}^2(z))dz\right)\\
\leq~&Cr^{n-2s}+C\sum_{i=2}^{[\log_2\left(\frac{r}{4R}\right)]+1}2^{i(n-2s)}R^{n-2s} \leq Cr^{n-2s},
\end{aligned}
\end{equation}
where we used $n>2s$ and $r>1.$ Combining  \eqref{5.8}, \eqref{5.9} and  \eqref{5.14}, we deduce
\begin{equation}
\label{5.15}
\int_{B_r\setminus B_{4R}}|(-\D)^\frac{s}{2} u_\alpha|^2dx\leq Cr^{n-2s}.
\end{equation}
Then, the proof follows from \eqref{5.7}, \eqref{5.15} and $n>2s.$
\end{proof}
We now use Lemma \ref{le5.2} to prove weighted $L^2$-estimate for $|\nabla\of|$.
\begin{lemma}
\label{le5.3}
We have
\begin{equation}
\label{5.16}
\sum_\alpha\int_{B_r^{n+1}\cap\r} y^{1-2s}|\nabla\of^\l|^2 dxdy\leq C(r),\quad
\forall r>1,\l\geq1.
\end{equation}
\end{lemma}

\begin{proof}
Under the assumption that $\sum_\alpha f_\alpha=0$, it is straightforward to see that
\begin{equation}
\label{5.17}
\sum_\alpha f_\alpha^\l=0.
\end{equation}
From the relation $u_\alpha^\l=f_\alpha^\l-f_{\alpha+1}^\l$, one gets
\begin{equation}
\label{5.18}
f_\alpha^\l=\sum_{\beta=1}^{\alpha-1}u_\beta^\l-\frac{1}{Q}\sum_{\beta=1}^{Q-1}(Q-\beta)u_\beta^\l.
\end{equation}
Based on \eqref{5.18},  we can find constant $C$ such that	
\begin{equation}
\label{5.19}
\sum_\alpha\int_{B_r^{n+1}\cap\r}y^{1-2s}|\nabla\of|^2dxdy\leq C\sum_{\alpha=1}^{Q-1}\int_{B_r^{n+1}\cap\r}y^{1-2s}|\nabla\ou_\alpha|^2dxdy.
\end{equation}	
Therefore, it suffices to show that
\begin{equation}
\label{5.20}
\int_{B_r^{n+1}\cap\r}y^{1-2s}|\nabla\ou_\alpha|^2dxdy\leq C(r),~\forall r\geq1,~\l\geq1,\quad \alpha=1,\cdots,Q-1.
\end{equation}
We consider the following decomposition,
$$u_\alpha^\l=u_{\alpha,1}^\l+u_{\alpha,2}^\l,$$ where
\begin{align*}
u_{\alpha,1}^\l(x)&=c(n,s)\int_{\B}\left(\frac{1}{|x-z|^{n-2s}}-\frac{1}{(1+|z|)^{n-2s}}\right)\varphi(z)(2V_\alpha(z)-V_{\alpha-1}(z)-V_{\alpha+1}(z))dz+d_\alpha^\l,\\
u_{\alpha,2}^\l(x)&=c(n,s)\int_{\B}\left(\frac{1}{|x-z|^{n-2s}}-\frac{1}{(1+|z|)^{n-2s}}\right)(1-\varphi(z))(2V_\alpha(z)-V_{\alpha-1}(z)-V_{\alpha+1}(z))dz,
\end{align*}
for $\varphi\in C_c^\infty(B_{4r})$  is such that $\varphi=1$ in $B_{2r}$. As in the proof of Lemma \ref{le5.2}, one can show that    $$\int_{\r}y^{1-2s}|\nabla\ou_{\alpha,1}^\l(x)|^2dxdy=\kappa_s\int_{\B}\left|\ss \ou_{\alpha,1}^\l(x)\right|^2dx\leq C(r).$$
Here and in what follows,  $\ou_{\alpha,i}^\l$ denotes the $s$-harmonic extension of $u_{\alpha,i}^\l$ for $i=1,2$. It remains to prove that
\begin{equation}
\label{5.21}
\int_{B_r^{n+1}\cap\r}y^{1-2s}|\nabla\ou^\l_{\alpha,2}(x)|^2dxdy\leq C(r)\quad \mbox{for every}~r\geq 1,~\lambda\geq1.
\end{equation}
Following the  arguments of Lemma \ref{le4.3}, one can verify that
\begin{equation}
\label{5.22}
\|\nabla u_{\alpha,2}^\l\|_{L^\infty(B_{3r/2})}\leq C(r) \ \ \text{and} \ \
\int_{\B}\frac{|u_{\alpha,2}^\l(x)|}{1+|x|^{n+2s}}dx\leq C(r),
\end{equation}
and consequently,
\begin{equation}
\label{5.23}
\|u_{\alpha,2}^\l\|_{L^\infty(B_{3r/2})}\leq C(r).
\end{equation}
To prove \eqref{5.21}, we  consider $\partial_y\ou^\l_{\alpha,2}$ and $\nabla_x\ou^\l_{\alpha,2}$ separately. For the first term, we notice that
\begin{align*}
\partial_y\ou^\l_{\alpha,2}(X)=~&\partial_y(\ou^\l_{\alpha,2}(x,y)-u^\l_{\alpha,2}(x))=\partial_y\int_{\B}P(X,z)(u_{\alpha,2}^\l(z)-u_{\alpha,2}^\l(x))dz
\\=~&d_{n,s}\partial_y\int_{\B}\frac{y^{2s}}{|(x-z,y)|^{n+2s}}(u_{\alpha,2}^\l(z)-u_{\alpha,2}^\l(x))dz\\
=~&d_{n,s}\int_{\B}\partial_y\left(\frac{y^{2s}}{|(x-z,y)|^{n+2s}}\right)(u_{\alpha,2}^\l(z)-u_{\alpha,2}^\l(x))dz,
\end{align*}
where
$$P(X,z)=d_{n,s}\frac{y^{2s}}{|(x-z,y)|^{n+2s}} \ \ \text{ and } \ \ d_{n,s}\int_{\B}\frac{y^{2s}}{|(x-z,y)|^{n+2s}}dz=1.$$
By \eqref{5.22},  for $|x|\leq r$ it holds that
\begin{equation}
\label{5.24}
\begin{aligned}
&\left|\int_{\B\setminus B_ {3r/2}}\partial_y\left(\frac{y^{2s}}{|(x-z,y)|^{n+2s}}\right) (u_{\alpha,2}^\l(z)-u_{\alpha,2}^\l(x)) dz\right|\\
&\leq C(r)y^{2s-1}\int_{\B\setminus B_ {3r/2}}\frac{(|u_{\alpha,2}^\l(z)|+1)}{1+|z|^{n+2s}}dz\leq C(r)y^{2s-1}.
\end{aligned}
\end{equation}
Using \eqref{5.22}-\eqref{5.23}, yields
\begin{equation}
\label{5.25}
\begin{aligned}
&\int_{B_r^{n+1}\cap\r}y^{1-2s}\left(\int_{B_{3r/2}}\partial_y\left(\frac{y^{2s}}{|(x-z,y)|^{n+2s}}\right)(u_{\alpha,2}^\l(z)-u_{\alpha,2}^\l(x))dz\right)^2dxdy\\
&\leq C(r)\|\nabla u_{\alpha,2}^\l\|_{L^\infty(B_{3r/2})}^2\int_{B_r^{n+1}\cap\r}y^{1-2s}
\left(\int_{B_{3r/2}}\partial_y\left(\frac{y^{2s}}{|(x-z,y)|^{n+2s}}\right)|x-z|dz\right)^2dxdy\\
&\leq C(r)\|\nabla u_{\alpha,2}^\l\|_{L^\infty(B_{3r/2})}^2\int_{B_r^{n+1}\cap\r} y^{1-2s} dxdy\leq C(r).
\end{aligned}
\end{equation}
From \eqref{5.24}-\eqref{5.25},  we get
\begin{equation}
\label{5.26}
\begin{aligned}
\int_{B_r^{n+1}\cap\r}y^{1-2s}|\partial_y\ou^\l_{\alpha,2}|^2dxdy
\leq C(r)\int_{B_r^{n+1}\cap\r}\left(y^{1-2s}+y^{2s-1}\right)dxdy\leq C(r).
\end{aligned}
\end{equation}
For the term $\nabla_x\ou_{\alpha,2}^\l$,  in a similar way,   we get
\begin{equation}
\label{5.27}
\int_{B_r^{n+1}\cap\r}y^{1-2s}|\nabla_x\ou^\l_{\alpha,2}|^2dxdy\leq C(r).
\end{equation}
Then,  \eqref{5.20} follows from \eqref{5.26} and \eqref{5.27}. This completes the proof.	
\end{proof}

\begin{proposition}
\label{pr5.1}
We have
\begin{equation*}
\lim_{\l\to+\infty}E_s(\overline{ \bff },0,\l)=\lim_{\l\to\infty}E_s(\overline{ \bff }^\l,0,1)<+\infty.
\end{equation*}
\end{proposition}

\begin{proof}
From Lemma \ref{le4.1} and Lemma \ref{le5.3}, we conclude that the following is bounded in $\l\in[1,\infty)$
\begin{equation*}
  \sum_\alpha\left(\frac12\int_{\r\cap B_1^{n+1}}y^{1-2s}|\nabla\of^\l|^2dxdy-\kappa_s\int_{\partial\r\cap B_1^{n+1}}e^{-(\of^\l-\ofa^\l)}dx\right).
\end{equation*}
Using Theorem \ref{thmono1s} and Lemma \ref{le5.1}, we get
\begin{equation*}
\begin{aligned}
E(\ou^\l,0,1)&=sc_s\sum_{\ell=1}^{[\frac{Q}{2}]}(\ell Q-\ell^2)(d_\ell^\l+d_{Q-\ell}^\l)+O(1)
\\&\geq  E(\ou,0,1)=sc_s\sum_{\ell=1}^{[\frac{Q}{2}]}(\ell Q-\ell^2)(d_\ell^1+d_{Q-\ell}^1)+O(1),
\end{aligned}
\end{equation*}
which together with Lemma \ref{le4.4} implies
\begin{equation}
\label{5.28}
\sum_{\ell=1}^{[\frac{Q}{2}]}(\ell Q-\ell^2)(d_\ell^\l+d_{Q-\ell}^\l) ~\ \mbox{is bounded for}~\ \l\geq1.
\end{equation}
In addition, using the fact that $(\ell Q-\ell^2)$ is strictly positive for every $\ell=1,\cdots,[\frac{Q}{2}]$, we  get
$$d_\alpha^\l ~\ \mbox{is bounded for}~\ \l\geq1,\quad \alpha=1,\cdots,Q-1.$$
From Lemma \ref{le5.1}, we conclude
\begin{equation}
\label{5.29}
\sum_\alpha(2\alpha-Q-1)\int_{\partial B^{n+1}_1\cap\r}y^{1-2s}\of^\l(X)d\sigma=O(1).
\end{equation}
This finishes the proof.
\end{proof}

\begin{lemma}
\label{le5.4}
For every $r>0$ and $\lambda\geq 1$, we have $$\int_{B_r^{n+1}\cap\r}t^{1-2s}\left(|\ou_\alpha^\l|^2+|\nabla \ou_\alpha^\l|^2\right)dxdy\leq C(r),\quad\alpha=1,\cdots,Q-1.$$
\end{lemma}
\begin{proof}
Based on Lemma \ref{le5.3} we only need to show that
\begin{equation*}
\int_{B_r^{n+1}\cap\r}y^{1-2s}|\ou_\alpha^\l|^2dxdy\leq C(r).
\end{equation*}	
Since  $d_\alpha^\l$ is bounded, from Lemma \ref{le33.12} we have
\begin{equation*}
\int_{B_r^{n+1}\cap\r}y^{1-2s}|\ou_\alpha^\l|^2dxdy\leq C\int_{B_r^{n+1}\cap\r}y^{1-2s}(1+|\overline{v}_\alpha^\l|^2)dxdy\leq C(r).
\end{equation*}
This finishes the proof.
\end{proof}

We are now ready to provide the proofs of Theorem \ref{thmtodam} and Theorem \ref{thmtodaq}. Since the proofs are very similar when $n>2s$, we provide the proof simultaneously.

\begin{proof}[Proofs of Theorem \ref{thmtodam} and Theorem \ref{thmtodaq}.] Let $n>2s$. Let $(f_1,\cdots,f_Q)$ be a stable outside a compact solution of \eqref{main} satisfying \eqref{assum2}. Similar arguments hold for stable solutions of  \eqref{main} satisfying \eqref{assum2}. Let $R>1$ be such that $\bff$ is stable outside the ball $B_R.$ From Lemma \ref{le5.4} and \eqref{5.18} we obtain that
$$\int_{B_r^{n+1}\cap\r}y^{1-2s}(|\of^\l|^2+|\nabla\of^\l|^2)dxdy\leq C(r),\quad \alpha=1,\cdots,Q.$$
Hence, there exists a sequence $\lambda_i\to+\infty$ such that $$\of^{\lambda_i}~\mbox{converges weakly to}~\of^\infty~\mbox{in}~ \dot{H}^1_{\mathrm{loc}}(\overline\r,y^{1-2s}dxdy),\quad\alpha=1,\cdots,Q.$$ In addition, we have $\ou_\alpha^{\l_i}\to \ou_\alpha^\infty$ almost everywhere. We can argue as \cite[Theorem 1.1]{hy} to show that $\of^\infty$ satisfies \eqref{maine} in the weak sense. Next, we show that the limit function $\of^\infty$ is homogenous. Based on the above convergence arguments, for any $r>0$ we get
\begin{equation}
\label{5.r}
{\lim_{i\to\infty}E(\overline{f},\l_ir,0)~\ \mbox{is independent of}~\ r}.
\end{equation}
Indeed, for any two positive numbers $r_1 <r_2$ we have
\begin{equation*}
\lim_{i\to\infty}E(\overline{f},\l_ir_1,0)\leq\lim_{i\to\infty}E(\overline{f},\l_ir_2,0).
\end{equation*}
On the other hand, for any $\lambda_i$, we  choose $\lambda_{m_i}$ such that $\{\lambda_{m_i}\}\subset\{\lambda_i\}$ and $\lambda_i r_2\leq \lambda_{m_i} r_1$. As a consequence, we have
\begin{equation*}
\lim_{i\to\infty}E(\overline{f},\l_ir_2,0)\leq\lim_{i\to\infty}E(\overline{f},\l_{m_i}r_1,0)=\lim_{i\to\infty}E(\overline{f},\l_ir_1,0).
\end{equation*}
This finishes the proof of \eqref{5.r}. Using \eqref{5.r} we see that for $R_2>R_1>0,$
\begin{equation*}
\begin{aligned}
0=~&\lim_{i\to\infty}E(\overline{f},\l_iR_2,0)-\lim_{i\to+\infty}E(\overline{f},\l_iR_1,0)\\
=~&\lim_{i\to\infty}E(\overline{f}^{\l_i},R_2,0)-\lim_{i\to\infty}E(\overline{f}^{\l_i},R_1,0)\\
\geq~&\liminf\limits_{i\to+\infty}\int_{\left(B_{R_2}^{n+1}\setminus B_{R_1}^{n+1}\right)\cap\r}y^{1-2s}\sum_\alpha\left(\frac{\partial \of^{\l_i}}{\partial r}-\frac{(2\alpha-Q-1)s}{r}\right)^2dxdy\\
\geq~&\int_{\left(B_{R_2}^{n+1}\setminus B_{R_1}^{n+1}\right)\cap\r}y^{1-2s}\sum_\alpha\left(\frac{\partial \of^{\infty}}{\partial r}-\frac{(2\alpha-Q-1)s}{r}\right)^2dxdy.
\end{aligned}
\end{equation*}
where we only used the weak convergence of $\of^{\l_i}$
to $\of^\infty$ in $H^1_{\mathrm{loc}}(\overline\r,y^{1-2s}dxdy)$ in last inequality. Therefore,
$$\frac{\partial \of^\infty}{\partial r}+\frac{s(Q+1-2\alpha)}{r}=0\quad \mbox{a.e. in}\quad \r.$$
In addition, $\bff^\infty$ is also stable because the stability condition for $\bff^{\l_i}$ passes to the limit. This is in contradiction with Theorem \ref{homog}.  Now, let  $n\le 2s$ and  $(f_1,\cdots,f_Q)$ be a stable solution of \eqref{main}.  Applying the stability inequality with appropriate test functions and using the $s$-extension arguments, completes the proof.
\end{proof}

\end{document}